\documentclass{amsart}

\usepackage{amssymb,amsmath,color}
\usepackage{amsthm}
\usepackage{url}
\usepackage{tikz}
\usepackage[enableskew]{youngtab}
\usepackage{ytableau,varwidth
}
\allowdisplaybreaks

\def\({\left(}
\def\){\right)}

\newtheorem{theorem}{Theorem}[section]
\newtheorem{corollary}[theorem]{Corollary}

\theoremstyle{definition}

\newtheorem{example}{Example}
\newtheorem{remark}{Remark}

\begin{document}

\title[]{
Refinements of Beck-type partition identities}
\author{T. Amdeberhan}
\address{Department of Mathematics\\ Tulane University\\ New Orleans, LA 70118, USA}
\email{tamdeber@tulane.edu}

\author{G. E. Andrews}
\address{Department of Mathematics\\ Penn State University\\ University Park, PA 16802, USA}
 \email{gea1@psu.edu} 

\author{C. Ballantine}
\address{Department of Mathematics and Computer Science\\ College of the Holy Cross \\ Worcester, MA 01610, USA \\} 
\email{cballant@holycross.edu}

%\subjclass xxx \endsubjclass

\begin{abstract} Franklin's identity generalizes Euler's identity and states that the number of partitions of $n$ with $j$ different parts divisible by $r$ equals the number of partitions of $n$ with $j$ repeated parts. In this article, we give a refinement of Franklin's identity when $j=1$. We prove Franklin's identity when $j=1$, $r=2$ for partitions with fixed perimeter, i.e., fixed largest hook. We also derive a Beck-type  identity  for partitions with fixed perimeter: the excess in the number of parts in all partitions into odd parts with perimeter $M$ over the  number of parts in all partitions into distinct parts with perimeter $M$ equals the number of partitions with perimeter $M$ whose set of even parts is a singleton.  We provide analytic and combinatorial proofs of our results.
\end{abstract}

\maketitle

\section{Introduction}

\noindent
A {\it partition}  of $n\in\mathbb{N}$ is a finite non-increasing sequence of positive integers that add up to $n$. We use the notation $\lambda=(\lambda_1,\lambda_2,\dots,\lambda_{\ell})$, with $\lambda_1\geq\lambda_2\geq \ldots, \lambda_\ell$ and 
$\lambda_1+\lambda_2+\cdots+\lambda_{\ell}=n$. We refer to the numbers  $\lambda_i$ as the {\it parts} of $\lambda$. The {\it size} of $\lambda$, denoted  by
$\vert\lambda\vert$,  is the sum of all its parts.  We write $\lambda\vdash n$ to mean $\lambda$ is a partition of size $n$.   The {\it length} of $\lambda$, denoted by $\ell(\lambda)$, is the number of parts of $\lambda$. For a non-negative integer $n$, 
we denote by $p(n)$ the number of partitions of $n$. Moreover, we write $p(n\mid X)$ for the number of partitions of $n$ satisfying condition $X$. For more details about partitions, we refer the reader to \cite{A98}. 

\smallskip
\noindent
A classical partition identity is an identity of the form $p(n \mid X)=p(n \mid Y)$. One of the oldest partition identity is Euler's identity $$p(n \mid \text{parts are odd})=p(n \mid \text{parts are distinct}).$$ Glaisher's generalization of  Euler's identity states that, for $r\geq 2$,
$$p(n \mid \text{parts are not divisible by $r$})=p(n \mid \text{parts are repeated less than $r$ times}).$$ 
We note that partitions with parts not divisible by $r$ are called {\it $r$-regular partitions}. Franklin generalized Glaisher's identity further. To state the identity, we denote by $\mathcal O(n;r,j)$ the set of partitions of $n$ in which exactly $j$ different parts are divisible by $r$ and these parts can be repeated. We also denote by $\mathcal D(n;r,j)$ the set of partitions of $n$ in which exactly $j$ parts occur at least $r$ times. Then, Franklin's identity states that, for $r\geq 2$ and any non-negative integer $j$, $$|\mathcal O(n;r,j)|=|\mathcal D(n;r,j)|.$$
For additional discussions on partitions and partition identities, again refer to \cite{A98}.

\smallskip
\noindent
In this article, we prove a refinement of Franklin's identity for $j=1$. For a positive integer $u$, we denote by $\mathcal{O}_u(n;r,1)$ the set of partitions $\lambda \in \mathcal{O}(n;r,1)$ such that the unique part of $\lambda$ divisible by $r$ is repeated exactly $u$ times. Moreover, we denote by $\mathcal{D}_u(n;r,1)$ the set of partitions $\lambda \in \mathcal{D}(n;r,1)$ such that the unique part occuring at least $r$ times equals $u$. Let $\alpha_u^{(r)}(u)=|\mathcal{O}_u(n;r,1)|$ and $\beta^{(r)}(u)=|\mathcal{D}_u(n;r,1)|$.
\begin{theorem} \label{T1} Let $u, r$ be integers such that $u \geq 1$ and $r\geq 2$. Then, for any non-negative integer $n$, we have  $$\alpha_u^{(r)}(n)=\beta_u^{(r)}(n).$$ \end{theorem}

\smallskip
\noindent
In 2017, Beck \cite{Beck} conjectured  and Andrews \cite{A17} proved a companion identity to Euler's identity relating the number of parts in all partitions involved in Euler's identity. 
Denote by $a(n)$ the number of parts in all partitions in $\mathcal{O}(n;2,0)$, i.e., partitions of $n$ into odd parts,  and by $b(n)$ the number of parts in all partitions in $\mathcal{D}(n;2,0)$, i.e., partitions of $n$ into distinct parts. 
\begin{theorem}[Andrews] If $n$ is a non-negative integer, we have $$a(n)-b(n)=|\mathcal{O}(n;2,1)|=|\mathcal{D}(n;2,1)|.$$
\end{theorem}

\smallskip
\noindent
A Beck-type companion identity to Glaisher's identity was provided by Yang \cite{Y18}. Ballantine and Welch \cite{BW} gave Beck-type companion identities to Franklin's identity.

\smallskip
\noindent
In \cite{S}, Straub discovered that Euler's identity holds for partitions of fixed perimeter rather than fixed size. 
In \cite{CL}, the \it perimeter \rm of a partition $\lambda$ is defined to be the maximum hook length in the Ferrers diagram of $\lambda$, i.e.,  $\lambda_1+\ell(\lambda)-1$. 
\begin{theorem}[Straub]\label{euler} The number of partitions into distinct parts with perimeter $M$ is equal to the number of partitions into odd parts with perimeter $M$. Both are enumerated by the Fibonacci number $F_M$.   \end{theorem}
\smallskip
\noindent
Motivated by Straub's theorem and the Beck-type companion identities toward generalizations of Euler's theorem, we study several partition numbers for certain partitions of fixed perimeter and for the number of parts in these sets of partitions. 

\smallskip
\noindent
First, let's fix some notations.  We denote $\mathcal{G}(M)$ (respectively $\mathcal{H}(M)$)  the set of all partitions into odd parts (respectively distinct parts)  with perimeter $M$. Let $g(M)$ be the total number parts in all partitions in $\mathcal{G}(M)$, and $h(M)$ be the total number of parts in all partitions in $\mathcal{H}(M)$. Moreover, we denote by  $\mathcal{G}_1(M)$ the set of partitions with perimeter $M$ in which all parts are odd except one even part (possibly repeated). We also denote by $\mathcal{H}_1(M)$ the set of partitions with perimeter $M$ where one part occurs at least twice and all other parts occur once. We let $g_1(M)=|\mathcal{G}_1(M)|$ and $h_1(M)=|\mathcal{H}_1(M)|$.

 \begin{example} We list the first few terms of each sequence starting with $M=1$. 
\begin{enumerate} \item  $g(M)$ starts with: $1, 2, 4, 8, 15, 28, 51, 92,\dots$.\smallskip 
\item  $h(M)$ starts with: $1, 1, 3, 5, 10, 18, 33, 59,\dots$.\smallskip
\item  $g_1(M)$ starts with: $0, 1, 2, 5, 10, 20, 38, 71,\dots$.\smallskip
 \item $h_1(M)$ starts with: $0, 1, 2, 5, 10, 20, 38, 71,\dots$.\smallskip
 \item $g(M)-h(M)$ starts with: $0, 1, 1, 3, 5, 10, 18, 33,\dots$.
\end{enumerate}
\end{example}

\smallskip
\noindent The generating functions for these sequences are given in the next theorem and its corollary. 

\newpage

\begin{theorem}  \label{T6} The generating functions for the sequences $g(M), h(M), g_1(M)$, and  $h_1(M)$ are 
\begin{align*} \sum_{M\geq1}g(M)\,x^M&=\frac{x-x^3}{(1-x-x^2)^2}, \qquad \sum_{M\geq1}h(M)\,x^M=\frac{x-x^2}{(1-x-x^2)^2}, \\
 \sum_{M\geq1}g_1(M)\,x^M&=\frac{x^2}{(1-x-x^2)^2}=\sum_{M\geq1}h_1(M)\,x^M.  \end{align*} \end{theorem}
\smallskip
\noindent In the proof of this theorem, we also furnish formulas for these sequences.

 \begin{corollary}
\begin{align*}\sum_{M\geq1}(g_1(M)-g_1(M-1))\,x^M&=\frac{x^2-x^3}{(1-x-x^2)^2} =\sum_{M\geq1}(h_1(M)-h_1(M-1))\,x^M, \\
\sum_{M\geq1}(g(M)-h(M))\,x^M&=\frac{x^2-x^3}{(1-x-x^2)^2}.\end{align*}\end{corollary}

\smallskip
\noindent Next, we provide recurrence relations for $g(M), h(M), g_1(M)$, and  $h_1(M)$.

\begin{theorem} \label{T5} For all $M\geq 2$, we have
\begin{align}
\label{r1} g(M)&=g(M-1)+g(M-2)+F_{M-1}, \\
\label{r2} h(M)&=h(M-1)+h(M-2)+F_{M-2}, \\
\label{r3} g_1(M)&=g_1(M-1)+g_1(M-2)+F_{M-1}, \\
\label{r4} h_1(M)&=h_1(M-1)+h_1(M-2)+F_{M-1}.
\end{align}
 \end{theorem}

\smallskip
\noindent The next corollary follows immediately from Theorem \ref{T5}. It can also be proven by making use of the generating functions in Theorem \ref{T6}. We adopt the convention that  $F_{-1}=1, F_0=0, F_1=1$.

\begin{corollary} \label{Cfib} For $M\geq 2$, we have
\begin{enumerate}
\item $g(M)=F_M+\sum_{k=1}^{M-1}F_kF_{M-k}$; \medskip 
\item $h(M)=F_M+\sum_{k=1}^{M-2}F_kF_{M-1-k}$;\medskip
\item $g_1(M)=h_1(M)=\sum_{k=1}^{M-1}F_kF_{M-k}$;\medskip
\item $g(M)-h(M)= g_1(M)-g_1(M-1)$  

\hspace{1.8cm}$ = h_1(M)-h_1(M-1)=F_{M-1}+\sum_{k=1}^{M-3}F_kF_{M-2-k}.$ \end{enumerate}\end{corollary}
 
\smallskip
\noindent Part (4) of Corollary \ref{Cfib} leads to a Beck-type companion identity to  Straub's  theorem.

 \begin{corollary} \label{g-h} $g(M)-h(M)$ equals the number of partitions in $\mathcal G_1(M)$ with no  part equal to $1$. 
 \end{corollary}
 
 \begin{remark} The identity $g_1(M)=h_1(M)$ in part (3) of Corollary \ref{Cfib}  is Franklin's identity for $j=1$, $r=2$ applied to partitions of fixed perimeter. 
 \end{remark}

\newpage

 \section{Partitions with fixed perimeter} 
 
\noindent The Ferrers diagram of a partition $\lambda=(\lambda_1, \lambda_2, \ldots, \lambda_{\ell})$ is an array of left-justified boxes such that the $i$-th row from the top contains $\lambda_i$ boxes. For example, the Ferrers diagram of the partition $\lambda= (6,6, 3, 2, 2,1)$ is shown below. $$\tiny\ydiagram{6,6, 3, 2, 2,1}$$ We can also describe a partition $\lambda$ by a binary string or bit string  $b(\lambda)$ starting with $1$ and ending with $0$ as follows. Start at SW corner of the Ferrers diagram  and travel along the outer profile going NE. The profile is a Dyck path. For each step to the right, record a $1$; for each step up, record a $0$. For example, the bit string representing the partition $\lambda= (6,6, 3, 2, 2,1)$ is $b(\lambda)=101001011100$.
The perimeter of the partition $\lambda$ is one less than the number of digits in $b(\lambda)$. The number of parts of $\lambda$ equals the number of $0$ digits in $b(\lambda)$. This interpretation, with $N$ and $E$ in place of $0$ and $1$ respectively, is used by Fu and Tang in \cite{FT} to prove theorems for partitions with fixed perimeter.

 \smallskip 
\noindent
 Let $t_n(M)$ be the number of partitions of $n$ with perimeter $M$.  Since partitions having perimeter $M$ are in bijection with bit strings of length $M+1$ where the first digit is $1$ and last digit is $0$ and  there are $2^{M-1}$ such bit strings, there are $2^{M-1}$ partitions with perimeter $M$.

\begin{example}
\begin{align*} t_5(5)=& |\{(5), (4,1), (3,1,1), (2,1,1,1), (1,1,1,1,1)\}|=5,\\ t_6(5)=& |\{(4,2), (3,2,1), (2,2,1,1)\}|=3,\\  t_7(5)= & |\{(4,3), (3,3,1), (3,2,2), (2,2,2,1)\}|=4,\\ t_8(5)=& |\{(4,4), (3,3,2), (2,2,2,2)\}|=3,\\  t_9(5)=& |\{(3,3,3)\}|=1\end{align*} and $t_n(5)=0$ if $n\geq 10$. Thus, there are $16=2^4$ partitions with perimeter $5$.\end{example}

\smallskip 
\noindent The generating function for $t_n(M)$ is given in the next theorem. It likely appears in the literature but we were unable to find a mention of it.

\begin{theorem} We have
$$\sum_{n,M\geq0}t_n(M)\,z^Mq^n=\sum_{M\geq0}z^Mq^M\sum_{j=0}^{M-1}\begin{bmatrix} M-1\\j\end{bmatrix}_q.$$ \end{theorem}

\begin{proof} If $p(N,j,n)$ denotes the number of partitions of $n$ into at most $j$ parts, each less than or equal to $N$, then from \cite{A98} we have $$\begin{bmatrix} M-1\\j\end{bmatrix}_q=\sum_{n=0}^\infty p(M-j-1,j,n)q^n.$$ 

\noindent
Given $\lambda=(\lambda_1, \lambda_2, \ldots, \lambda_\ell)$, a partition of $n$ with  $\ell\leq j$ parts, each less than or equal to $M-j-1$, let   $\mu$ be  the partition obtained from $\lambda$ by adding an outer  hook of arm-length $M-j$ and leg-length $j$, i.e., $\mu=(M-j, \lambda_1+1, \lambda_2+1, \ldots \lambda_\ell+1, 1, \ldots, 1$, where the multiplicity of $1$ is $j-\ell$ (and it could be $0$). Then $\mu$ has perimeter $M$ and size $n+M$. 
\end{proof}

\noindent 
We note that Staub's bijection in his proof  of Theorem \ref{euler} is defined recursively. Fu and Tang \cite{FT} give a new interpretation of Straub's  bijection that is no longer recursive.  We describe their bijection $\psi: \mathcal{H}(M)\to \mathcal{G}(M)$ in the language of bit strings.
Note that the bit string of a partition with distinct parts has no repeated $0$'s. The bit string for a partition with odd parts has the property that every initial substring ending in $0$ contains an odd number of digits equal to $1$. So,  given $\lambda \in \mathcal{H}(M)$, the bit string $b(\lambda)$   is  transformed (working from left to right) as follows: 

\smallskip
The first $1$  remains unchanged. \smallskip 

A digit $1$ that is preceded by a $1$ changes to $0$. \smallskip

A substring $01$ changes to $11$. \smallskip

The last $0$ remains unchanged. \smallskip

\noindent The resulting bit string  $b(\mu)$ defines  a partition $\mu\in \mathcal{G}(M)$. Define $\psi(\lambda)=\mu$.

\section{Proof of Theorem \ref{T1}}

\begin{proof}[Analytic proof] We have
\begin{align*} \sum_{n,u\geq0}^{\infty}\alpha_u^{(r)}(n)z^uq^n&=\sum_{j=1}^{\infty}\frac{zq^{rj}}{1-zq^{rj}}\,\prod_{k=1}^{\infty}\frac{1-q^{rk}}{1-q^k} \\
&=\sum_{j,h\geq1}(zq^{rj})^h\,\prod_{k=1}^{\infty}\frac{1-q^{rk}}{1-q^k} \\
&=\frac1{(q;q)_{\infty}}\sum_{h\geq1}z^hq^{rh}\prod_{\substack{k=1 \\ k\neq h}}(1-q^{rk}) \\
&=\sum_{h\geq1}\frac{z^hq^{rh}}{1-q^h}\prod_{\substack{k=1 \\ k\neq h}}(1+q^r+q^{2r}+\cdots+q^{(k-1)r}) \\
&= \sum_{n,u\geq0}^{\infty}\beta_u^{(r)}(n)z^uq^n. 
\end{align*}\end{proof}

\begin{proof}[Combinatorial proof] We employ the exponential notation and write $(a^m)$ for the partition with $m$ parts all equal to $a$. Given two partitions $\eta$ and $\xi$, the partition $\eta\cup \xi$ is the partition obtained by rearranging the parts of $\eta$ and $\xi$, respecting multiplicities.
We construct a bijection $$\psi: \mathcal{O}_u(n;r,1)\to \mathcal{D}_u(n;r,1).$$ Let $\lambda\in \mathcal{O}_u(n;r,1)$ and suppose $d=rj$, $j\geq 1$, is the part divisible by $r$. Write $\lambda=\mu\cup ((rj)^u)$, where $\mu\in \mathcal{O}(n;r,0)$. We define $\psi(\lambda)=\varphi(\mu)\cup (u^{rj})\in \mathcal{D}_u(n;r,1)$ where $\varphi$ is Glasher's bijection. Note that $\varphi(\mu)\in \mathcal{D}(n;r,0)$ could have parts equal to $u$. These occur less than $r$ times. 

\smallskip
\noindent
To see that $\psi$ is invertible, let $\eta\in \mathcal{D}_u(n;r,1)$ and suppose that part $u$ occurs $m_u\geq r$ times in $\eta$. Express $m_u=qr+s$, for unique  $q\geq 1$,  $0\leq s<r$, and write $\eta=\xi\cup  (u^{rs})$. We have $\xi\in  \mathcal{D}(n;r,0)$. Then $\psi^{-1}(\eta)=\varphi^{-1}(\xi)\cup ((rs)^u)\in \mathcal{O}_u(n;r,1)$ with $\varphi^{-1}(\xi)\in \mathcal{O}(n;r,0)$.
\end{proof}

\section{An Analytic proof of Theorem \ref{T6}}

\noindent In this section, we prove Theorem \ref{T6} and show how it leads to an analytic proof of Theorem \ref{T5}.

\begin{proof} [A formula for $g(M)$]  The generating function for partitions with odd parts where the power of $q$ is the number partitioned, the power of $z$ is the perimeter and the power of $t$ is the number of parts, takes the form
$$\sum_{n\geq0}\frac{tz^{2n+1}q^{2n+1}}{(tzq;q^2)_{n+1}}=\sum_{n\geq0}tz^{2n+1}q^{2n+1}\sum_{s\geq0}\begin{bmatrix}n+s\\n\end{bmatrix}_{q^2}(tzq)^s.$$
The coefficient of $z^M$ in the above sum is
$$\sum_{n\geq0}q^Mt^{M-2n}\begin{bmatrix}{M-n-1}\\n\end{bmatrix}_{q^2}.$$
To obtain a formula for  $g(M)$, we must set $q=1$, differentiate with respect to $t$ and then set $t=1$. The outcome is
\begin{equation}\label{A} g(M)=\sum_{0\leq 2n\leq M-1}(M-2n)\binom{M-n-1}n. \end{equation}
Formula \eqref{A} allows an easy computation of the generating function for $g(M)$. Namely,
\begin{align*} \sum_{M\geq0}g(M)\,z^M&=\sum_{M\geq0}z^M\sum_{0\leq 2n\leq M-1}(M-2n)\binom{M-n-1}n \\
&=\sum_{M,n\geq0}z^{M+2n+1}(M+1)\binom{M+n}n\\& =\sum_{M\geq0}{z^{M+1}(M+1)}{(1-z^2)^{-(M+1)}} 
=\sum_{M\geq1}{z^MM}{(1-z^2)^{-M}} \\
&=\frac{\frac{z}{1-z^2}}{\left(1-\frac{z}{1-z^2}\right)^2}
=\frac{z-z^3}{(1-z-z^2)^2}. 
\end{align*} \end{proof}

\begin{proof} [A formula for $h(M)$]  We proceed exactly as was done for $g(M)$. The three variable generating function is
$$\sum_{n\geq0}(-tzq;q)_nz^{n+1}q^{n+1}t=\sum_{n=0}^{\infty}\sum_{j=0}^n\begin{bmatrix}{n}\\j\end{bmatrix}_q q^{\binom{j+1}{2}}t^jz^jz^{n+1}q^{n+1}t.$$
The coefficient of $z^M$ in the sum above is
$$\sum_{0\leq 2j\leq M-1}t^{j+1}q^{M+\binom{j}2}\begin{bmatrix}{M-j-1}\\j\end{bmatrix}_{q}.$$
To obtain $h(M)$, we must set $q=1$, differentiate with respect to $t$ and then set $t=1$. Hence
\begin{equation}\label{B} h(M)=\sum_{0\leq 2j\leq M-1}(j+1)\binom{M-j-1}j. \end{equation}
Formula \eqref{B} facilitates an easy computation of the generating function for $h(M)$. That is,
\begin{align*} \sum_{M\geq0}h(M)\,z^M&=\sum_{M\geq0}z^M\sum_{0\leq 2j\leq M-1}(j+1)\binom{M-j-1}j \\
&=\sum_{M,j\geq0}z^{M+2j+1}(j+1)\binom{M+j}j=\sum_{j\geq0}{z^{2j+1}(j+1)}{(1-z)^{-(j+1)}} \\
&=\sum_{j\geq1}{z^{2j-1}j}{(1-z)^{-j}}=z^{-1}\frac{\frac{z^2}{1-z}}{\left(1-\frac{z^2}{1-z}\right)^2}
=\frac{z-z^2}{(1-z-z^2)^2}.
\end{align*} \end{proof}

\begin{proof} [A formula for $g_1(M)$]  Recall that $g_1(M)$ is the number of perimeter $M$ partitions with odd parts and one exception, i.e., exactly one part is even and this part may be repeated. In the bivariate generating function below, the power of $q$ is the number partitioned and  the power of $z$ is the perimeter. We write the generating function as two sums. In the first, the largest part is even; in the second, the largest part is odd:
$$\sum_{k\geq1}\frac{z^{2k}q^{2k}}{(zq;q^2)_k(1-zq^{2k})}+\sum_{k\geq0}\frac{z^{2k+1}q^{2k+1}}{(zq;q^2)_{k+1}}\sum_{j=1}^k\frac{zq^{2j}}{1-zq^{2j}}.$$
Setting $q=1$, we obtain the generating function for $g_1(M)$:
\begin{align*} \sum_{M\geq0}g_1(M)\,z^M&=\sum_{k\geq1}\frac{z^{2k}}{(1-z)^{k+1}}+\sum_{k\geq0}\frac{z^{2k+2}k}{(1-z)^{k+2}} \\
&=\sum_{k\geq0}\left(\frac{z^{2k+2}}{(1-z)^{k+2}}+\frac{z^{2k+2}k}{(1-z)^{k+2}}\right) \\
&=\sum_{k\geq0}\sum_{j\geq0}z^{2k+2+j}\binom{k+1+j}j(k+1).
\end{align*}

\noindent
Therefore, we find that
$$g_1(M)=\sum_{k\geq0}(k+1)\binom{M-k-1}{k+1}.$$
We may also simplify the generating function for $g_1(M)$ so that 
\begin{align*} \sum_{M\geq0}g_1(M)\,q^M&=\sum_{M\geq1}\frac{z^{2M}}{(1-z)^{M+1}}+\sum_{M\geq0}\frac{z^{2M+2}M}{(1-z)^{M+2}} \\
&=\frac{z^2}{(1-z)(1-z-z^2)}+\frac{z^4}{(1-z)^3\left(1-\frac{z^2}{1-z}\right)^2} 
=\frac{z^2}{(1-z-z^2)^2}. 
\end{align*} \end{proof}

\begin{proof} [A formula for $h_1(M)$] Recall that $h_1(M)$ is the number of perimeter $M$ partitions with distinct parts and one exception, i.e., exactly one part is repeated. Again, when setting up the bivariate generating function, we must take into account whether or not (i) the largest part is repeated or (ii) the largest part is {\it not} repeated:
$$\sum_{M\geq1}\frac{z^{M+1}q^{2M}}{1-zq^M}\,(-zq;q)_{M-1}+\sum_{M\geq 2}z^Mq^M\sum_{j=1}^{M-1}\frac{(-zq;q)_{M-1}}{1+zq^j}\,\frac{z^2q^{2j}}{1-zq^j}.$$

\noindent
Therefore, setting $q=1$ in the sum above, we obtain

\begin{align*} \sum_{M\geq0}h_1(M)\,q^M&=\sum_{M\geq1}\frac{z^{M+1}(1+z)^{M-1}}{1-z}+\sum_{M\geq2}\sum_{j=1}^{M-1}\frac{z^{M+2}(1+z)^{M-2}}{1-z} \\
&=\sum_{M\geq1}\frac{z^{M+1}(1+z)^{M-1}}{1-z}+\sum_{M\geq2}\frac{(M-1)z^{M+2}(1+z)^{M-2}}{1-z} \\
&=\sum_{M\geq1}\frac{z^{M+1}(1+z)^{M-1}}{1-z}+\sum_{M\geq1}\frac{Mz^{M+3}(1+z)^{M-1}}{1-z} \\
&=\sum_{M\geq1}\frac{z^{M+1}(1+z)^{M-1}(1+z^2M)}{1-z} \\
&=\frac{z}{1-z^2}\sum_{M\geq1}(z(1+z))^M(1+z^2M) \\
&=\frac{z(z+z^2)}{(1-z^2)(1-z-z^2)}+\frac{z^3}{1-z^2}\cdot\frac{z+z^2}{(1-z-z^2)^2} \\
&=\frac{z^2}{(1-z-z^2)^2} \\
&=\sum_{M\geq0}g_1(M)\,z^M.
\end{align*}
Hence, we have found the generating function and proved that $h_1(M)=g_1(M)$. \end{proof}
\medskip

\noindent
\begin{remark}  The four assertions in Theorem \ref{T5} now follow for the following reason.
\smallskip
\noindent
Each of the sequences $g(M), h(M), g_1(M), h_1(M)$ satisfy the same fourth order recurrence relation because each generating function has identical denominator $(1-z-z^2)^2$. In addition,
\begin{align*} \sum_{M\geq0}F_M\,z^M&=\frac{z(1-z-z^2)}{(1-z-z^2)^2}, \\
\sum_{M\geq0}MF_M\,z^M& =\frac{z+z^3}{(1-z-z^2)^2}, \\
\sum_{M\geq0}\sum_{j=0}^MF_jF_{M-j}\,z^M&=\frac{z^2}{(1-z-z^2)^2}.
\end{align*}
Hence, all the sequences in the four claims of Theorem \ref{T5} fulfil the same fourth order recurrence. Consequently, the proof of the four claims only require checking that they are true for $M\leq 4$. 

\end{remark}

\section{A combinatorial proof of Theorem \ref{T5}}

\noindent
In the current section, we supply combinatorial arguments for the four recurrence relations listed in Theorem \ref{T5}.

\begin{proof}
\noindent We first prove the recurrence relation \eqref{r1} for $g(M)$: 
$$ g(M)=g(M-1)+g(M-2)+F_{M-1}.$$
Recall that if $\lambda\in \mathcal{G}(M)$, then every initial substring of $b(\lambda$ ending in $0$ contains an odd number of digits equal to $1$. The set $\mathcal{G}(M)$ can be partitioned so that $\mathcal{G}(M)=\mathcal{G}'(M)\sqcup \mathcal{G}''(M)$, where $\mathcal{G}'(M)$ consists of partitions $\lambda \in \mathcal{G}(M)$ such that the second digit of $b(\lambda)$ is $0$ and $\mathcal{G}''(M)$ consists of partitions $\lambda \in \mathcal{G}(M)$ such that the second digit of $b(\lambda)$ is $1$.

\smallskip
\noindent
 Given $\lambda \in \mathcal{G}'(M)$,  we remove the second digit of $b(\lambda)$, which is $0$,  to obtain the bit string $b(\mu)$ of a partition $\mu \in \mathcal{G}(M-1)$. Moreover, $\ell(\mu)=\ell(\lambda)-1$. This gives a bijection between $\mathcal{G}'(M)$ and 
$\mathcal{G}(M-1)$.  Thus, the total number of parts in all partitions in $\mathcal{G}'(M)$ equals $g(M-1)+\vert\mathcal{G}(M-1)\vert=g(M-1)+F_{M-1}$.

\smallskip
\noindent
 Given $\lambda \in \mathcal{G}''(M)$, since the second digit in $b(\lambda)$ is $1$, the third digit must also be $1$. We remove the second and third digits of $b(\lambda)$, which are both equal to $1$,  to obtain the bit string $b(\mu)$ of a partition $\mu \in \mathcal{G}(M-2)$. Moreover, $\ell(\mu)=\ell(\lambda)$. As a result, we obtain a length-preserving bijection between $\mathcal{G}''(M)$ and $\mathcal{G}(M-2)$.  The total number of parts in all partitions in $\mathcal{G}''(M)$ equals $g(M-2)$. So, equation \eqref{r1} holds for all $M\geq 2$. 
 
\bigskip 
\noindent
Next,  we prove the recurrence relation \eqref{r2} for $h(M)$:
$$ h(M)=h(M-1)+h(M-2)+F_{M-2}. $$
Recall that if $\lambda\in \mathcal{H}(M)$, then $b(\lambda)$ has no repeated $0$'s.
The set $\mathcal{H}(M)$ can be partitioned into $\mathcal{H}(M)=\mathcal{H}'(M)\sqcup \mathcal{H}''(M)$, where $\mathcal{H}'(M)$ consists of partitions $\lambda \in \mathcal{H}(M)$ such that the second digit of $b(\lambda)$ is $1$, and $\mathcal{H}''(M)$ consists of partitions $\lambda \in \mathcal{H}(M)$ such that the second digit of $b(\lambda)$ is $0$. 

\smallskip
\noindent
Given $\lambda \in \mathcal{H}'(M)$, we remove the second digit of $b(\lambda)$, which is $1$, to obtain the bit string $b(\mu)$ of a partition $\mu \in \mathcal{H}(M-1)$. Moreover, $\ell(\mu)=\ell(\lambda)$. This gives a length-preserving bijection between $\mathcal{H}'(M)$ and 
$\mathcal{H}(M-1)$.  Thus, the total number of parts of all partitions in $\mathcal{H}'(M)$ equals $h(M-1)$.

\smallskip

\noindent
Given $\lambda \in \mathcal{H}''(M)$, since the second digit is $0$, the third digit must be  $1$. We remove the second and third digits from $b(\lambda)$ to obtain the bit string $b(\mu)$ of a partition $\mu \in \mathcal{H}(M-2)$. Moreover, $\ell(\mu)=\ell(\lambda)-1$. This lends itself to a bijection between $\mathcal{H}''(M)$ and 
$\mathcal{H}(M-2)$.  The total number of parts of all partitions in $\mathcal{H}''(M)$ equals $h(M-2)+\vert \mathcal{H}''(M)\vert= h(M-2)+\vert\mathcal{H}(M-2)\vert= h(M-2)+F_{M-2}$. Therefore, equation \eqref{r2} holds for all $M\geq 2$.

\bigskip 
\noindent
Next,  we prove the recurrence relation \eqref{r3} for $g_1(M)$:
$$g_1(M)= g_1(M-1)+g_1(M-2)+F_{M-1}. $$
Recall that  $\mathcal{G}_1(M)$ is the set of partitions with perimeter $M$ and all parts odd except one even part (possibly repeated). 
Thus, if $\lambda \in \mathcal{G}_1(M)$, then $b(\lambda)$
 has a \textit{unique}  initial substring ending in $0$ and containing an even number of $1$'s.  For example, if $b(\lambda)=1110100$, then $\lambda \in \mathcal{G}_1(6)$ and the unique initial substring ending in $0$ with an even number of $1$'s is $11101$.   If  $b(\lambda)=1110000$, then there is no such initial substring. Thus, $\lambda$ has no even part and $\lambda \not\in \mathcal{G}_1(6)$. If   $b(\lambda)=111010110$, then there are two such substrings, $111010$ and $111010110$. Thus, $\lambda$ has two even parts and $\lambda \not\in \mathcal{G}_1(8)$. 

\medskip

\noindent
We partition $\mathcal{G}_1(M)$ into $\mathcal{G}_1(M)= \mathcal{G}'_1(M)\sqcup \mathcal{G}''_1(M)\sqcup \mathcal{G}'''_1(M)$, where $\mathcal{G}'_1(M)$ consists of partitions 
$\lambda \in \mathcal{G}_1(M)$ such that the second digit of $b(\lambda)$ is $0$, $\mathcal{G}''_1(M)$ consists of partitions $\lambda \in \mathcal{G}_1(M)$ such that the second and third digits of $b(\lambda)$ are both $1$'s, and $\mathcal{G}'''_1(M)$ consists of partitions $\lambda \in \mathcal{G}_1(M)$ such that the second digit of $b(\lambda)$ is $1$ and third digit is $0$. 
\medskip

\noindent
If $\lambda \in \mathcal{G}'_1(M)$, removing the second digit from $b(\lambda)$ we obtain the bit string $b(\mu)$ of a partition $\mu \in \mathcal{G}_1(M-1)$. This gives a bijection between $\mathcal{G}'_1(M)$ and $\mathcal{G}_1(M-1)$. Similarly, the transformation that removes the second and third digits from a bit string corresponding to a partition in $\mathcal{G}''_1(M)$ gives a bijection between $\mathcal{G}''_1(M)$ and $\mathcal{G}_1(M-2)$. Next, given a partition $\lambda \in \mathcal{G}'''_1(M)$, we denote by  $b_0(\lambda)$ the maximal substring of $0$'s starting with the third digit of $b(\lambda)$. Suppose $b_0(\lambda)$ has length $k$. Then $1\leq k\leq M-1$. 
If  $k<M-1$ then  removing the substring $b_0(\lambda)$ from $b(\lambda)$ we obtain a bit string $b(\mu)$ for a partition $\mu\in \mathcal{G}''(M-k)$, i.e.,  $b(\mu)\in \mathcal{G}(M-k)$ with third digit equal to $1$. As shown above, $\mathcal{G}''(M-k)$ is in bijection with $\mathcal{G}(M-k-2)$. There is exactly one partition $\lambda$ with $b_0(\lambda)$ of length $M-1$ and we can map this partition to the empty partition $\emptyset$. Thus, $\mathcal{G}'''_1(M)$ is in bijection with $\mathcal{G}(M-3)\sqcup\mathcal{G}(M-4)\sqcup \cdots \sqcup \mathcal{G}(1)\sqcup\{\emptyset\}$. Since $\vert\mathcal{G}(N)\vert=F_N$ for all $N$, we have 
\begin{align*} g_1(M)=& g_1(M-1)+g_1(M-2)+ F_{M-3}+F_{M-4}+\cdots +F_1+1\\  = & g_1(M-1)+g_1(M-2)+F_{M-1}. \qquad
\end{align*}

\bigskip 
\noindent
Finally,  we prove the recurrence relation \eqref{r4} for $h_1(M)$:
$$h_1(M)= h_1(M-1)+h_1(M-2)+F_{M-1}. $$
Recall that $\mathcal{H}_1(M)$ is the set of partitions with perimeter $M$ with one part repeated and all other parts distinct. If $\lambda \in \mathcal{H}_1(M)$ then either 
\begin{itemize}
\item[(i)] $b(\lambda)$ has a single substring $\bar b_0(\lambda)$ of length $2\leq k\leq M-2$ with all digits $0$'s having a prefix and suffix $1$ in $b(\lambda)$, or\smallskip

\item[(ii)] $b(\lambda)=1000\dots0$, i.e. in $b(\lambda)$  all digits except the first are equal to $0$. 

\end{itemize}

\noindent
We partition $\mathcal{H}_1(M)$ into $\mathcal{H}_1(M)= \mathcal{H}'_1(M)\sqcup \mathcal{H}''_1(M)\sqcup \mathcal{H}'''_1(M)$ where $\mathcal{H}'_1(M)$ consists of partitions 
$\lambda \in \mathcal{H}_1(M)$ such that the second digit of $b(\lambda)$ is $1$, $\mathcal{H}''_1(M)$ consists of partitions $\lambda \in \mathcal{H}_1(M)$ such that the second digit of $b(\lambda)$ is $0$ and the third digit is $1$, and $\mathcal{H}'''_1(M)$ consists of partitions $\lambda \in \mathcal{H}_1(M)$ such that the second and third digits of $b(\lambda)$ are both $0$. 
\medskip

\noindent
If $\lambda \in \mathcal{H}'_1(M)$, removing the second digit from $b(\lambda)$ we obtain the bit string $b(\mu)$ of a partition $\mu \in \mathcal{H}_1(M-1)$. This gives a bijection between $\mathcal{H}'_1(M)$ and $\mathcal{H}_1(M-1)$. Similarly, the transformation that removes the second and third digits from a bit string corresponding to a partition in $\mathcal{H}''_1(M)$ gives a bijection between $\mathcal{H}''_1(M)$ and $\mathcal{H}_1(M-2)$. Next, consider a partition $\lambda \in \mathcal{H}'''_1(M)$.  If $\lambda$ is in case (i) above, removing the substring $\bar b_0(\lambda)$ from $b(\lambda)$ we obtain a bit string $b(\mu)$ for a partition $\mu\in \mathcal{H}'(M-k)$, i.e.,  $b(\mu)\in \mathcal{H}(M-k)$ with second digit equal to $1$. As shown above, $\mathcal{H}'(M-k)$ is in bijection with $\mathcal{H}(M-k-1)$. Lastly, there is exactly one partition $\lambda$ in the case (ii) above, and we can map this partition to the empty partition $\emptyset$. Thus, $\mathcal{H}''_1(M)$ is in bijection with $\mathcal{H}(M-3)\sqcup\mathcal{H}(M-4)\sqcup \cdots \sqcup \mathcal{H}(1)\sqcup\{\emptyset\}$. Since $\vert\mathcal{H}(N)\vert=F_N$ for all $N$, we have 
\begin{align*} h_1(M)=& h_1(M-1)+h_1(M-2)+ F_{M-3}+F_{M-4}+\cdots +F_1+1\\  = & h_1(M-1)+h_1(M-2)+F_{M-1}. 
\end{align*}
 This completes the proof of  Theorem \ref{T5}. \end{proof}

\begin{remark} Theorem \ref{T5} implies Theorem \ref{T6}. We explain this for the sequence $h(M)$.  Let $E$ be the  \textit{forward shift operator}  given by $Ew(M)=w(M+1)$. In general, if a function $w(M)$ is annihilated by the linear operator $L:=a_0+a_1E+\cdots+a_rE^r$ then it is true that the ordinary generating function for $w(M)$ is a rational function whose denominator is the polynomial $a_r+\cdots+a_1x^{r-1}+a_0x^r$. The recurrence  \eqref{r2} can be reformulated as $(E^2-E-1)h(M-2)=F_{M-2}$. On the other hand, the Fibonacci sequence satisfies $(E^2-E-1)F_{M-2}=0$. It is now immediate 
$$(E^2-E-1)^2h(M)=(E^2-E-1)(E^2-E-1)h(M-2)=(E^2-E-1)F_{M-2}=0.$$
 In light of this, we may write
$$\sum_{M\geq1}h(M)x^M=\frac{b_0+b_1x+b_2x^2+b_3x^3}{(1-x-x^2)^2}$$
for  constants $b_0, b_1, b_2$ and $b_3$, which in turn are determined from  initial conditions. Since $h(0)=0, h(1)=1, h(2)=1, h(3)=3$, one can easily check that $$\sum_{M\geq1}h(M)\,x^M=\frac{x-x^2}{(1-x-x^2)^2},$$ as stated in Theorem \ref{T6}. \end{remark}

\section{Interpretation of $g(M)-h(M)$}

\bigskip
\noindent
Define the \it index of a partition $\lambda$ \rm to be $\ell(\lambda)-1-\lfloor \lambda_1/2\rfloor$, i.e., 
$$\text{(\# of parts of $\lambda$)} -1-\left\lfloor\frac{\text{largest part of $\lambda$}}2\right\rfloor.$$

 \begin{remark}Note that for a partition $\lambda$ into odd parts, the index of $\lambda$ is precisely the $M2$-rank of $\lambda$ as defined by Berkovich and Garvan \cite{BG}.  The $M2$-rank of a partition $\lambda$ is the rank of the partition $\mu$ whose Ferrers diagram has the same shape as the $2$-modular diagram of $\lambda$. \end{remark}
\noindent
\begin{theorem} \label{T7} $g(M)-h(M)$ is the sum of indices of all the partitions into odd parts with perimeter $M$. 
\end{theorem}

\noindent
\begin{example} $g(4)-h(4)=8-5=3$. The partitions in question are $3+1, 3+3$ and $1+1+1+1$ whose indices are $0, 0$ and $3$, respectively. Thus the sum of the indices is indeed equal to $3$.\end{example}

\newpage

\noindent
\begin{proof} The generating function for partitions with odd parts where the power of $q$ is the number partitioned, the power of $z$ is the perimeter and the power of $t$ is the index, takes the form 
$$\sum_{n\geq0}\frac{t^{-n}z^{2n+1}q^{2n+1}}{(tzq;q^2)_{n+1}}
=\sum_{n\geq0}t^{-n}z^{2n+1}q^{2n+1}\sum_{s\geq0}\begin{bmatrix} n+s\\n\end{bmatrix}_{q^2}(tzq)^s.$$
The coefficient of $z^M$ in the above is
$$\sum_{n\geq0}q^Mt^{M-3n-1}\begin{bmatrix} M-n-1\\n\end{bmatrix}_{q^2}.$$
To obtain the sum of all the indices, we must set $q=1$, differentiate with respect to $t$ and then set $t=1$. The result is the finite sum

\begin{align*}
& \sum_{n\geq0}(M-3n-1)\binom{M-n-1}n \\
= & \sum_{n\geq0}(M-2n)\binom{M-n-1}n-\sum_{n\geq0}(n+1)\binom{M-n-1}n \\
= & g(M)-h(M). 
\end{align*}
\end{proof}

\begin{remark}It follows from Theorem \ref{T7} and Corollary \ref{g-h} that the sum of the indices of all partitions of fixed perimeter with odd parts is non-negative.\end{remark} 

\section{Final remarks and questions}

\noindent
For $r\geq 2$, let $\mathcal{G}_r(M)$ be the set of all perimeter $M$ partitions with parts not divisible by $r$ and set $g_r(M)=|\mathcal{G}_r(M)|$. Let  and $\mathcal{H}_r(M)$  be the set of all perimeter $M$ partitions with parts occurring $<r$ times and set $h_r(M)=|\mathcal{H}_r(M)|$. It is easily checked that, for $r\geq 3$, the analogue of Glaisher's identity for partitions with fixed perimeter fails. Below, we find the generating functions for $h_r(M)$ and $g_r(M)$, as well as the generating functions for the number of partitions with fixed perimeter and parts congruent to $d$ modulo $r$. 

\noindent
\begin{theorem}\label{T7a} For $r\geq 2$, we have a generating function for $h_r(M)$ given by
$$\sum_{M\geq1}h_r(M)\,z^M=\frac{z+z^2+\cdots+z^{r-1}}{1-z-z^2-\cdots-z^r}.$$ \end{theorem}

\begin{proof} In the bivariate generating function below, the exponent of $z$ keeps track of the perimeter and the exponent of $q$ keeps track of the size of the partition. If $h_r(n,M)$ denotes the number of partitions of $n$ with perimeter $M$ and having parts repeated less that $r$ times, then 
\begin{align*}& \sum_{n\geq0} \sum_{M\geq1}h_r(n,M) z^Mq^n=\\  &  \sum_{n\geq 1}z^{n-1}(zq^n+z^2q^{2n}+\cdots+ z^{r-1}q^{(r-1)n})\prod_{j=1}^{n-1}(1+zq^j+z^2q^{2j}+\cdots +z^{r-1}q^{(r-1)j}).\end{align*}

\newpage

\noindent Setting $q=1$ we obtain 

\begin{align*}\sum_{M\geq1}h_r(M)z^M &= \sum_{n\geq 1}z^{n-1} \cdot \frac{z-z^{r}}{1-z}\cdot \left(\frac{1-z^{r}}{1-z}\right)^{n-1}\\ & = \frac{z-z^{r}}{1-z}\cdot\frac{1}{1-\frac{z(1-z^r)}{1-z}}\\ & = \frac{z-z^r}{1-2z+z^{r+1}}.
\end{align*}
Since $1-z$ divides both the numerator and the denominator, we cancel and find $$\sum_{M\geq1}h_r(M)\,z^M=\frac{z+z^2+\cdots+z^{r-1}}{1-z-z^2-\cdots-z^r}.$$ 
\end{proof}

\begin{theorem} \label{T7b} For $r\geq 2$, we have a generating function for $g_r(M)$ given by
$$\sum_{M\geq1}g_r(M)\,z^M=\frac{z}{1-2z}\cdot\frac{(1-z)^{r-1}-z^{r-1}}{((1-z)^{r-1}-z^r)}.$$
\end{theorem}

\begin{proof} First, we keep track of the size of the partitions with exponent $q$ so that there comes a generating function
$$\sum_{n\geq0}\sum_{d=1}^{r-1}z^{nr+d}q^{nr+d}\prod_{j=1}^{nr+d}\frac1{1-zq^j}\prod_{i=1}^n(1-zq^{ir}).$$
Now, set $q=1$ to get
\begin{align*} \sum_{M\geq1}g_r(M)\,z^M&=\sum_{n\geq0}\sum_{d=1}^{r-1}\frac{z^{nr+d}}{(1-z)^{n(r-1)+d}} 
=\sum_{d=1}^{r-1}\left(\frac{z}{1-z}\right)^d\sum_{n\geq0}\left(\frac{z^r}{(1-z)^{r-1}}\right)^n \\
&=\frac{z}{1-z}\frac{1-\left(\frac{z}{1-z}\right)^{r-1}}{1-\frac{z}{1-z}} \frac1{1-\frac{z^r}{(1-z)^{r-1}}}.
\end{align*}
This simplifies to the desired formulation as depicted in the theorem.
\end{proof}

\begin{theorem} \label{T8} Let $r\geq 2$ and $1\leq d<r$.  Denote by $\mathcal{G}_r^{(d)}(M)$  the set of partitions with perimeter $M$ and having parts congruent to $d$ modulo $r$, and let's denote $g_r^{(d)}(M)=|\mathcal{G}_r^{(d)}(M)|$. Then,  the generating function for $g_r^{(d)}(M)$ is 
$$\sum_{M\geq1}g_r^{(d)}(M)\,z^M=\frac{z^d}{1-z-z^r}.$$ 
In particular, $g_r^{(d)}(M)=g_r^{(d-1)}(M-1)$ for each $1\leq d<r$. \end{theorem}

\begin{proof} We mimic  the proof of the generating function for $g(M)$ in Theorem \ref{T5}. In the bivariate generating function below, the exponent of $z$ keeps track of the perimeter and the exponent of $q$ keeps track of the size of the partition. Then, if $g_r^{(d)}(n,M)$ denotes the number of partitions of $n$ with perimeter $M$ having parts congruent to $d$ modulo $r$, we have 

$$\sum_{n\geq0} \sum_{M\geq1}g_r^{(d)}(n,M)z^Mq^n=\sum_{n\geq0}z^{nr+d}q^{nr+d}\sum_{s\geq0}\begin{bmatrix}n+s\\n\end{bmatrix}_{q^r}(zq)^s.$$ 
\medskip

\noindent Setting $q=1$ we obtain 
\begin{align*}\sum_{M\geq1}g_r^{(d)}(M)z^M &= z^d\sum_{s\geq0}z^{s}\sum_{n\geq0}\binom{n+s}{n}(z^r)^n 
=  z^d\sum_{s\geq0}\frac{z^s}{(1-z^r)^{s+1}} \\
&=\frac{z^d}{1-z^r}\cdot \frac{1}{1-\frac{z}{1-z^r}}=\frac{z^d}{1-z-z^r}.\end{align*}
\end{proof}
\begin{remark} If $r=2$, from Theorems \ref{T7a} and \ref{T8} we obtain $g_2^{(1)}(M)=h_2(M)$, which is  Theorem \ref{euler}. 
\end{remark}

\smallskip
\noindent
We conclude by noting that, for fixed $M\geq 1$ and $r\geq 2$, one should be able to verify that  $g_r(M)\leq h_r(M)$  using  the  generating functions from Theorems \ref{T7a} and \ref{T7b}; for instance, by comparing the largest real roots of the denominator polynomials. However, this is likely tedious and not very illuminating. Instead, it would be more appealing to give a positive answer to the following question. 

\medskip
\noindent
\bf Question. \it  Is there a combinatorial proof that $g_r(M)\leq h_r(M)$  for all integers $M\geq 1$ and $r\geq 2$? \rm

\end{document}